\documentclass[preprint,10pt]{elsarticle}

\usepackage{tikz}
\usepackage{hyperref}
\usepackage{amsmath,amssymb,amsthm}
\usepackage{dsfont}  

\newcommand{\supp}{\mathrm{supp\,}} 

\makeatletter
\def\moverlay{\mathpalette\mov@rlay}
\def\mov@rlay#1#2{\leavevmode\vtop{%
   \baselineskip\z@skip \lineskiplimit-\maxdimen
   \ialign{\hfil$\m@th#1##$\hfil\cr#2\crcr}}}
\newcommand{\charfusion}[3][\mathord]{
    #1{\ifx#1\mathop\vphantom{#2}\fi
        \mathpalette\mov@rlay{#2\cr#3}
      }
    \ifx#1\mathop\expandafter\displaylimits\fi}
\makeatother
\newcommand{\cupdot}{\charfusion[\mathbin]{\cup}{\cdot}}

\newcommand{\C}{\mathbb C}
\newcommand{\R}{\mathbb R}

\newcommand{\Z}{\mathbb{Z}}
\newcommand{\N}{\mathbb{N}}

\newcommand{\kernel}{\mathrm{ker}} 
\newcommand{\dist}{\mathrm{dist}} 

\def\boldb{ \boldsymbol{b}}
\def\bc{ \boldsymbol{c}}

\def\bv{ \boldsymbol{v}}

\newtheorem*{thm-nonumber}{Theorem}			
\newtheorem{theorem}{Theorem}

\newtheorem{lemma}[theorem]{Lemma}
\newtheorem*{prop-nonumber}{Proposition}			
\newtheorem{proposition}[theorem]{Proposition}
\newtheorem{example}[theorem]{Example}
\newtheorem*{example-nonumber}{Example}
\newtheorem*{remark}{Remark}		
\newtheorem{definition}[theorem]{Definition}         

\newtheorem{Question}{Question}

\makeatletter
\newcommand{\neutralize}[1]{\expandafter\let\csname c@#1\endcsname\count@}
\makeatother

\usepackage{color}

\begin{document}

\begin{frontmatter}



\title{Unions of exponential Riesz bases}


\author{Dae Gwan Lee}
\ead{daegwans@gmail.com; daegwan@kumoh.ac.kr}


\address{Department of Mathematics and Big Data Science, Kumoh National Institute of Technology, \\
Gumi, Gyeongsangbuk-do 39177, Republic of Korea (South Korea)}




\begin{abstract}
We develop new methods for constructing exponential Riesz bases by taking unions of exponential Riesz bases.  
These methods are based on taking unions of frequency sets and domains respectively and therefore allow easy construction. 
Together with examples that illustrate our methods, we also provide several examples showing the delicate nature of exponential Riesz bases.
\end{abstract}




\end{frontmatter}



\section{Introduction and Main Results}
\label{sec:intro}

For a discrete set $\Lambda \subset \R$ and a positive measure set $S \subset \R$, we consider the \emph{complex exponential system} $E(\Lambda) := \{e^{2\pi i \lambda (\cdot)} : \lambda\in\Lambda\}$ on $S$.
The set $\Lambda$ is often called a \emph{frequency set} or a \emph{spectrum}, and $S$ is called a \emph{domain}.

Despite the simple formulation of exponential systems, the existence of exponential Riesz bases for a given domain $S$ is a highly nontrivial problem.
In 1995, Seip \cite{Se95} proved that for any interval $S \subset [0,1)$, there exists a set $\Lambda \subset \Z$ such that $E(\Lambda)$ is a Riesz basis for $L^2(S)$.
Only recently in 2015, Kozma and Nitzan \cite{KN15} showed that for any finite union $S$ of disjoint intervals in $[0,1)$,
there exists a set $\Lambda \subset \Z$ such that $E(\Lambda)$ is a Riesz basis for $L^2(S)$.
Up to date, there are only a handful of classes of domains for which exponential Riesz bases are known to exist, see e.g., \cite{CC18,CL22,DL19,GL14,Ko15,KN15,LPW22,LR00,PRW21}.
Recently, Kozma, Nitzan, and Olevskii \cite{KNO21} constructed a bounded measurable set $S \subset \R$ such that no exponential system can be a Riesz basis for $L^2(S)$.

Aside from the question of existence of exponential Riesz bases,
one may ask whether frequency sets and domains of (already existing) exponential Riesz bases can be combined/split in a canonical way.

\begin{Question}[Combining exponential Riesz bases]
\label{Q1}
Let $S_1, S_2 \subset \R$ be disjoint measurable sets and
let $\Lambda_1, \Lambda_2 \subset \R$ be discrete sets with $\dist (\Lambda_1, \Lambda_2) := \inf \{ | \lambda_1 - \lambda_2 | : \lambda_1 \in \Lambda_1, \, \lambda_2 \in \Lambda_2 \} > 0$. Assume that for each $n = 1,2$, the system $E(\Lambda_n)$ is a Riesz basis for $L^2(S_n)$.
Is $E(\Lambda_1 \cup \Lambda_2 )$ then a Riesz basis for $L^2( S_1 \cup S_2 )$?
\end{Question}

Note that we have excluded the case $\dist (\Lambda_1, \Lambda_2) = 0$ from consideration.
This case can be divided into two sub-cases:
\begin{itemize}
\item[(i)] $\Lambda_1 \cap \Lambda_2 \neq \emptyset$
\item[(ii)] $\Lambda_1 \cap \Lambda_2 = \emptyset$ and there exist sequences $\{ \lambda_{1,k} \}_{k=1}^\infty \subset \Lambda_1$ and $\{ \lambda_{2,k} \}_{k=1}^\infty \subset \Lambda_2$ such that $| \lambda_{1,k} - \lambda_{2,k} | < 1/k$ for all $k =1, 2, \cdots$. 
\end{itemize}
Certainly, we do not want the frequency sets $\Lambda_1$ and $\Lambda_2$ to intersect, since those elements lying in the intersection would be counted only once in $\Lambda_1 \cup \Lambda_2$. 
On the other hand, (ii) would imply that $\Lambda_1 \cup \Lambda_2$ is a non-separated set, in which case $E(\Lambda_1 \cup \Lambda_2 )$ is not even a Riesz sequence.
This justifies the assumption $\dist (\Lambda_1, \Lambda_2) > 0$ in Question~\ref{Q1}.

\begin{Question}[Splitting exponential Riesz bases]
\label{Q2}
Let $S_1, S_2 \subset \R$ be disjoint measurable sets.
If $\Lambda_1, \Lambda_2 \subset \R$ are disjoint sets such that $E(\Lambda_2)$ and $E(\Lambda_1 \cup \Lambda_2 )$ are Riesz bases for $L^2(S_2)$ and $L^2( S_1 \cup S_2 )$ respectively, then is the system $E(\Lambda_1)$ a Riesz basis for $L^2(S_1)$?
\end{Question}

Unfortunately, both questions turn out negative in general.

For Question~\ref{Q1}, note that $E( 4\Z \cup (4\Z{+}2) )$ is \emph{not} a Riesz basis for $L^2([0,\frac{1}{4}) \cup [\frac{2}{4},\frac{3}{4}))$,
even though $E(4\Z)$ and $E(4\Z{+}2)$ form Riesz bases for $L^2[0,\frac{1}{4})$ and $L^2[\frac{2}{4},\frac{3}{4})$ respectively.
However, it can be amended to form a Riesz basis by slightly shifting one of the frequency sets.
Indeed, for every $0 < \delta < 2$, the system $E( 4\Z \cup (4\Z{+}2{+}\delta) )$ is a Riesz basis for $L^2([0,\frac{1}{4}) \cup [\frac{2}{4},\frac{3}{4}))$ while each of $E(4\Z)$ and $E(4\Z{+}2{+}\delta)$ forms a Riesz basis for $L^2[0,\frac{1}{4})$ and $L^2[\frac{2}{4},\frac{3}{4})$ respectively.
See Example \ref{ex:RB-N4-periodic} below for more details.

We also provide an example which cannot be amended by shifting the frequency sets as above.
Fix any $0 < \epsilon < \frac{1}{4}$ and let 
\begin{equation}\label{eqn:example-Lambda1-Lambda2}
\Lambda^{(1)} := \{ \ldots , \, -6, \, -4, \, -2, \, 0 , \, 1 {+}\epsilon , \, 3 {+}\epsilon , \, 5 {+}\epsilon , \, \ldots \} ,
\quad \Lambda^{(2)} := - \Lambda^{(1)}   \, . 
\end{equation}
Then $E(\Lambda^{(1)})$, $E(\Lambda^{(2)})$ and $E(\Lambda^{(1)}\cup \Lambda^{(2)} )$ form Riesz bases for $L^2[0,\frac{1}{2})$, $L^2[\frac{1}{2},1)$ and $L^2[0,1)$ respectively, despite the fact that $\Lambda^{(1)} \cap \Lambda^{(2)} = \{ 0 \}$.
To separate the frequency sets apart, we consider a uniform shifting of $\Lambda^{(2)}$.
For any $0 < \delta < \epsilon$, we have 
$\dist ( \Lambda^{(1)}, \Lambda^{(2)} {+} \delta) = \delta > 0$ and each of $E(\Lambda^{(1)})$ and $E(\Lambda^{(2)} {+} \delta)$ forms a Riesz basis for $L^2[0,\frac{1}{2})$ and $L^2[\frac{1}{2},1)$ respectively.
However, their union $E(\Lambda^{(1)} \cup (\Lambda^{(2)} {+} \delta))$ is overcomplete in $L^2[0,1)$ and thus \emph{cannot} be a Riesz basis for $L^2[0,1)$.
See Example \ref{ex:TwoRieszBases-1-intersection} below for more details.

Concerning Question~\ref{Q2},
the sets $\Lambda_1 := \Lambda^{(1)} \backslash \{ 0 \}$ and $\Lambda_2 := \Lambda^{(2)}$ are disjoint and satisfy that $E(\Lambda_2)$ and $E(\Lambda_1 \cup \Lambda_2 )$ are Riesz bases for $L^2[\frac{1}{2},1)$ and $L^2[0,1)$ respectively. However, $E(\Lambda_1)$ is incomplete in $L^2[0,\frac{1}{2})$ and therefore is not a Riesz basis for $L^2[0,\frac{1}{2})$.

These observations suggest that in order to derive some positive answers to Questions \ref{Q1} and \ref{Q2},
we would need to impose some additional conditions on the sets $S_1, S_2 \subset \R$ and $\Lambda_1, \Lambda_2 \subset \R$.

\subsection{Related work on Questions \ref{Q1} and \ref{Q2}}
\label{subsec:RelatedWork}

Recently, Frederick and Mayeli \cite{FM21} obtained the following result which is related to Question~\ref{Q1}.
Let $S \subset [0,1)$ be a measurable set and let $A \subset \Z_N := \{ 0, 1, \ldots, N{-}1 \}$ with $N \in \N$.
If there exist sets $\Lambda \subset \R$ and $J \subset \Z_N$ such that $E(\Lambda)$ and $E(J)$ are Riesz bases for $L^2(S)$ and $\ell_2 (A)$ respectively, and if
\begin{equation}\label{eqn:eqn5-FM21}
\widehat{\delta_\lambda} (\omega) = 1
\quad \text{for all} \;\; \lambda \in \Lambda , \;\; \omega \in A,
\end{equation}
then $E(\Lambda{+}\frac{J}{N})$ is a Riesz basis for $L^2(S{+}A)$ (see \cite[Theorem 2]{FM21}).
Here, the sums are understood in the Minkowski sense, that is, $\Lambda+\frac{J}{N} = \{ \lambda + \frac{j}{N} : \lambda \in \Lambda, \; j \in J \}$ and
$S+A = \{ s + a : s \in S, \; a \in A \}$. 
Also, $\delta_\lambda$ denotes the Dirac delta at $\lambda$, and $\widehat{f}$ denotes the Fourier transform of $f$ which is defined as $\widehat{f} (\omega) = \int_{-\infty}^{\infty} f(t) \, e^{- 2 \pi i t \omega} \, dt$ for $f \in L^1(\R) \cap L^2(\R)$ 
and is defined also for $f$ in $L^2(\R)$ and even tempered distributions by continuity. 

This result, however, relies heavily on the Fourier duality and the condition \eqref{eqn:eqn5-FM21} which requires some implicit symmetry in $\Lambda$ and $A$.

Another work related to Question~\ref{Q1} is the key lemma used by Kozma and Nitzan \cite{KN15}. We defer the detailed discussion to Section~\ref{sec:comparison-existing-work}.

Regarding Question~\ref{Q2}, we mention two related results.
It is shown in \cite[Proposition 2.1]{MM09} (also see \cite[Proposition 5.4]{BCMS19}) that for any $\Lambda \subset \Z$ and any measurable set $S \subset [0,1)$, the system $E(\Lambda)$ is a Riesz basis for $L^2(S)$ if and only if $E(\Z \backslash \Lambda)$ is a Riesz basis in $L^2 ( [0,1)\backslash S )$.
Also, it is proved in \cite{LS01} that if $E(\Lambda)$ with $\Lambda \subset \R$ is a Riesz basis for $L^2[0,1)$,
then for each $0 < a < 1$, there \emph{exists} a subset $\Lambda' \subset \Lambda$ such that $E(\Lambda')$ is a Riesz basis for $L^2[0,a)$ and $E(\Lambda \backslash \Lambda')$ is a Riesz basis for $L^2[a,1)$.
The former result provides a positive answer to Question~\ref{Q2} in the case that $\Lambda_1 \cupdot \Lambda_2 = \Z$ and $S_1 \cupdot S_2 = [0,1)$, while the latter result only shows the existence of a \emph{particular} splitting of $\Lambda$ into $\Lambda_1 := \Lambda'$ and $\Lambda_2 := \Lambda \backslash \Lambda'$ such that $E(\Lambda_1)$ is a Riesz basis for $L^2[0,a)$ and $E(\Lambda_2)$ is a Riesz basis for $L^2[a,1)$.

\subsection{Main Results}
\label{subsec:Main results}

In this paper, we develop new methods for combining exponential Riesz bases, which rely on 
taking unions of frequency sets and domains respectively.
This is conceptually different from the key lemmas of \cite{CL22} and \cite{KN15} which combines the frequency sets for $A_{\geq n}$, $n=1,\ldots,N$ (see Section \ref{sec:comparison-existing-work} for more details).
Also, our result requires only some mild conditions on the sets themselves and does not require any involved conditions like \eqref{eqn:eqn5-FM21}.


\begin{theorem}\label{thm:two-sets-nested-translation-by-a}
Let $S_1, S_2 \subset \R$ be disjoint bounded measurable sets with $S_1 {+} a \subset S_2$ for some $a \in \R \backslash \{ 0 \}$.
If there exist sets $\Lambda_1 \subset \R \backslash \frac{1}{a}\Z$ and $\Lambda_2 \subset \frac{1}{a}\Z$ with $\dist (\Lambda_1, \frac{1}{a}\Z) > 0$ such that for each $\ell = 1,2$, the system $E(\Lambda_\ell)$ is a Riesz basis (resp.~a complete sequence, a Riesz sequence) for $L^2(S_\ell)$, then $E ( \Lambda_1 \cup \Lambda_2 )$ is a Riesz basis (resp.~a complete sequence, a Riesz sequence) for $L^2( S_1 \cup S_2 )$.
\end{theorem}

As a direct generalization of Theorem \ref{thm:two-sets-nested-translation-by-a}, we have the following.

\begin{theorem}\label{thm:RB-union-two-sets-generalization}
Let $S_1, S_2 \subset \R$ be disjoint bounded measurable sets satisfying
$S_1 {+} ka \subset S_2$ for $k = 1, \ldots, K$, with some $a \in \R \backslash \{ 0 \}$ and $K \in \N$.
Let $c_1, \ldots, c_K \in [0,\frac{1}{a})$ be distinct real numbers, and assume that
there exist sets $\Lambda_1 \subset \R \backslash \cup_{k=1}^K (\frac{1}{a}\Z{+}c_k)$ and $\Lambda_2 \subset \cup_{k=1}^K (\frac{1}{a}\Z{+}c_k)$ with $\dist (\Lambda_1, \cup_{k=1}^K (\frac{1}{a}\Z{+}c_k)) > 0$, such that
for each $\ell = 1,2$, the system $E(\Lambda_\ell)$ is a Riesz basis (resp.~a complete sequence, a Riesz sequence) for $L^2(S_\ell)$.
Then $E ( \Lambda_1 \cup \Lambda_2 )$ is a Riesz basis (resp.~a complete sequence, a Riesz sequence) for $L^2( S_1 \cup S_2 )$.
\end{theorem}

It should be noted that $\Lambda_1 \subset \R \backslash \cup_{k=1}^K (\frac{1}{a}\Z{+}c_k)$ can be an {\it arbitrary} set which is not contained in a union of lattices.
%

We point out that the main advantage of Theorem \ref{thm:RB-union-two-sets-generalization}
lies in the \emph{arbitrary} choice of $c_1, \ldots, c_K \in [0,\frac{1}{a})$. For instance, if $S_1,S_2 \subset \R$ satisfy $S_1 {+} \frac{K}{N} \subset S_2$ for some $K, N \in \N$, then one needs $\Lambda_2 \subset \frac{N}{K}\Z = \cup_{k=0}^{K-1} (N\Z{+}k\frac{N}{K})$ in order to apply Theorem \ref{thm:two-sets-nested-translation-by-a}.
However, according to Theorem \ref{thm:RB-union-two-sets-generalization}, such a requirement can be weakened to $\Lambda_2 \subset \cup_{k=1}^K (N\Z{+}c_k)$ with \emph{arbitrary} $c_1, \ldots, c_K \in [0,N)$, if the conditions $S_1 {+} \frac{k}{N} \subset S_2$, $k = 1, \ldots, K$, are satisfied.

Using Theorem \ref{thm:RB-union-two-sets-generalization}, we can construct some exponential Riesz bases which cannot be obtained from the existing methods, such as the combining trick of Kozma and Nitzan (see Lemma \ref{lem:KN15-Lemma2} below).

\begin{example-nonumber}
\rm
Let $S_2 = [\frac{1}{3}, 1)$ and $\Lambda_2 = (3\Z{+}c_1) \cup (3\Z{+}c_2)$ with any
$c_1, c_2 \in [0,3)$ satisfying
\[
\det
\begin{bmatrix}
e^{- 2\pi i c_1 / 3} & e^{- 4\pi i c_1 / 3} \\
e^{- 2\pi i c_2 / 3} & e^{- 4\pi i c_2 / 3}
\end{bmatrix}
\neq 0 \, .
\]
Then $E(\Lambda_2)$ is a Riesz basis for $L^2(S_2)$ by Proposition \ref{prop:characterization-based-on-Fourier-submatrices} below.
Now, let $S_1 \subset [0, \frac{1}{3})$ be a measurable set and let $\Lambda_1 \subset \R \backslash \Lambda_2$ be a discrete set with $\dist (\Lambda_1, \Lambda_2) > 0$ such that $E(\Lambda_1)$ is a Riesz basis for $L^2(S_1)$.
Then the sets $S_1 {+} \frac{1}{3}$ and $S_1 {+} \frac{2}{3}$ are both contained in $S_2$, so the assumption of Theorem \ref{thm:RB-union-two-sets-generalization} is fulfilled with $a = \frac{1}{3}$ and $K = 2$.
Therefore, the system $E ( \Lambda_1 \cup \Lambda_2 )$ is a Riesz basis for $L^2( S_1 \cup S_2 )$ by Theorem \ref{thm:RB-union-two-sets-generalization}.

Note that there is no requirement on the structure of $\Lambda_1$. Indeed, we only require $\Lambda_1 \subset \R \backslash \Lambda_2$ to be a discrete set with $\dist (\Lambda_1, \Lambda_2) > 0$ such that $E(\Lambda_1)$ is a Riesz basis for $L^2(S_1)$. 
In contrast, Lemma \ref{lem:KN15-Lemma2} requires
all the frequency sets to be subsets of shifted copies of $N\Z$, hence, Lemma \ref{lem:KN15-Lemma2} is not applicable here.  
\end{example-nonumber}


\begin{figure}[h!]
\centering
\begin{tikzpicture}[x=9cm,y=1cm, scale=0.7]
  \begin{scope}
    \draw[-,thick] (0,0+4) -- (1,0+4); 
      \draw[thick] (0,0.15+4) --++ (0,-0.30)
        node[below=-1] {$0$};
      \draw[thick] (1/3,0.15+4) --++ (0,-0.30)
        node[below=-1] {$\frac{1}{3}$};
      \draw[thick] (2/3,0.15+4) --++ (0,-0.30)
        node[below=-1] {$\frac{2}{3}$};
      \draw[thick] (1,0.15+4) --++ (0,-0.30)
        node[below=-1] {$1$};

      \draw[thick] (1/15,0.08+4) --++ (0,-0.16);
      \draw[thick] (1/2/pi,0.08+4) --++ (0,-0.16);
      \draw[thick] (3/15,0.08+4) --++ (0,-0.16);
      \draw[thick] (4/15,0.08+4) --++ (0,-0.16);

\draw[red, line width = 1.5, -] (1/15,0.04+4) -- (1/2/pi,0.04+4);
\draw[red, line width = 1.5, -] (3/15,0.04+4) -- (4/15,0.04+4);

\draw[blue, line width = 1.5, -] (1/3,0.04+4) -- (1,0.04+4);

      \node[scale=1.2] at (1/15/2+1/2/pi/2+0.05, 0.6+4) {$S_1$};
      \node[scale=1.2] at (2/3,0.6+4) {$S_2$};


      \node[scale=1] at (1/15/2+1/2/pi/2+0.02, 0.4+2) {$\Lambda_1 \subset \R \backslash \Lambda_2$};

      \node[scale=1] at (2/3+0.07,0.4+2) {$\Lambda_2 = (3\Z{+}c_1) \cup (3\Z{+}c_2)$};

  \end{scope}
\end{tikzpicture}
\caption{An example of sets $S_1$ and $S_2$ together with $\Lambda_1$ and $\Lambda_2$}
\end{figure}
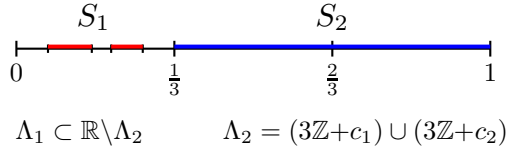

As a partial converse of Theorem \ref{thm:two-sets-nested-translation-by-a} and Theorem \ref{thm:RB-union-two-sets-generalization}, we have:

\begin{theorem}\label{thm:RB-union-two-sets-generalization-converse}
(a) Let $S_1, S_2 \subset \R$ be disjoint bounded measurable sets with $S_1 {+} a \subset S_2$ for some $a \in \R \backslash \{ 0 \}$.
If there exist sets $\Lambda_1 \subset \R \backslash \frac{1}{a}\Z$ and $\Lambda_2 \subset \frac{1}{a}\Z$ such that $E ( \Lambda_1 \cup \Lambda_2 )$ is complete in $L^2( S_1 \cup S_2 )$,
then $E(\Lambda_1)$ is complete in $L^2(S_1)$. \\
(b) Let $S_1, S_2 \subset \R$ be disjoint bounded measurable sets satisfying 
$S_1 {+} ka \subset S_2$ for $k = 1, \ldots, K$, with some $a \in \R \backslash \{ 0 \}$ and $K \in \N$.
Let $c_1, \ldots, c_K \in [0,\frac{1}{a})$ be distinct real numbers, and assume that
there exist sets $\Lambda_1 \subset \R \backslash \cup_{k=1}^K (\frac{1}{a}\Z{+}c_k)$ and $\Lambda_2 \subset \cup_{k=1}^K (\frac{1}{a}\Z{+}c_k)$
such that $E ( \Lambda_1 \cup \Lambda_2 )$ is complete in $L^2( S_1 \cup S_2 )$.
Then $E(\Lambda_1)$ is complete in $L^2(S_1)$.
\end{theorem}

Using Theorem \ref{thm:RB-union-two-sets-generalization-converse}(a), one can deduce the following.  
Let $S_1 = [0,\frac{1}{2})$ and $S_2 = [\frac{1}{2},1)$.
If there is a set $\Lambda_1 \subset \R \backslash 2\Z$ such that $E ( \Lambda_1 \cup 2\Z )$ is a Riesz basis for $L^2( S_1 \cup S_2 )$, then $E(\Lambda_1)$ is necessarily complete in $L^2(S_1)$ by Theorem \ref{thm:RB-union-two-sets-generalization-converse}(a).
This is in clear contrast with the counterexample for Question~\ref{Q2} discussed on p.\,2,
where the sets $\Lambda_1 := \Lambda^{(1)} \backslash \{ 0 \}$ and $\Lambda_2 := \Lambda^{(2)}$ are disjoint and satisfy that $E(\Lambda_2)$ and $E(\Lambda_1 \cup \Lambda_2 )$ are Riesz bases for $L^2(S_2)$ and $L^2( S_1 \cup S_2 )$ respectively, but $E(\Lambda_1)$ is incomplete in $L^2(S_1)$. 
Theorem \ref{thm:RB-union-two-sets-generalization-converse}(a) shows that such pathological examples cannot occur if $\Lambda_2 = 2 \Z$.


As a final remark, we point out that 
Theorems~\ref{thm:two-sets-nested-translation-by-a}, \ref{thm:RB-union-two-sets-generalization} and \ref{thm:RB-union-two-sets-generalization-converse} can be easily generalized to higher dimensions. 
For instance, Theorem~\ref{thm:two-sets-nested-translation-by-a} can be generalized to dimension $d$ as follows.

For $(a_1, \ldots, a_d) \in \R^d$, we define the set
\[
\begin{split}
\Gamma_{a_1, \ldots, a_d } &= \{ (x_1, \ldots, x_d ) \in \R^d : e^{2\pi i (a_1 x_1 + \cdots + a_d x_d )} = 1 \}  \\
&=  \{  (x_1, \ldots, x_d ) \in \R^d  : a_1 x_1 + \cdots + a_d x_d  \in \Z  \} 
\end{split}
\]
which is an additive closed subgroup of $\mathbb{R}^d$. 
For instance, in the $d=2$ case, 
\[
\Gamma_{0,0} =  \{ (x, y)  \in \R^2 : 0  x + 0   y \in \Z \}  = \R^2
\]
is the full $2$-dimensional plane, 
\[
\Gamma_{2,0} = \{ (x, y) : 2 x + 0  y \in \Z \} =  ( \tfrac{1}{2} \Z ) \times \R 
\]
is the union of all vertical lines whose $x$-intercept is a half-integer, and 
\[
\Gamma_{2,1} = \{ (x, y) : 2 x +  y \in \Z \}  = \{ (x, y) \in \R^2 : y = - 2x+ n, \; n \in \Z \}
\]
is the union of all lines having slope $-2$ and integer $y$-intercept.

\begin{theorem}\label{thm:two-sets-nested-translation-by-a-dimension-d}
Let $S_1, S_2 \subset \R^d$ be disjoint bounded measurable sets satisfying $S_1 {+} (a_1, \ldots, a_d) \subset S_2$ for some $(a_1, \ldots, a_d) \in \R^d \backslash \{ (0,\ldots,0) \}$. 
If there exist sets $\Lambda_1 \subset \R^d \backslash \Gamma_{a_1, \ldots, a_d }$ and $\Lambda_2 \subset \Gamma_{a_1, \ldots, a_d }$ with $\dist (\Lambda_1, \Gamma_{a_1, \ldots, a_d } ) > 0$ such that for each $\ell = 1,2$, the system $E(\Lambda_\ell)$ is a Riesz basis (resp.~a complete sequence, a Riesz sequence) for $L^2(S_\ell)$, then $E ( \Lambda_1 \cup \Lambda_2 )$ is a Riesz basis (resp.~a complete sequence, a Riesz sequence) for $L^2( S_1 \cup S_2 )$.
\end{theorem}



\subsection{Organization of the paper}
\label{subsec:organization-paper}

In Section \ref{sec:prelim}, we review some necessary background on exponential systems and give some examples.
In Section \ref{sec:comparison-existing-work}, we discuss some known methods for combining exponential Riesz bases and compare them with our results.
Sections \ref{sec:Proof-of-thm:two-sets-nested-translation-by-a}, \ref{sec:Proof-of-thm:RB-union-two-sets-generalization}, \ref{sec:Proof-of-prop:RB-union-two-sets-generalization-converse} and \ref{sec:Proof-of-thm:two-sets-nested-translation-by-a-dimension-d} are devoted to the proof of
Theorems \ref{thm:two-sets-nested-translation-by-a}, \ref{thm:RB-union-two-sets-generalization}, \ref{thm:RB-union-two-sets-generalization-converse} and \ref{thm:two-sets-nested-translation-by-a-dimension-d}, respectively.
Section \ref{sec:appendix-proof-prop-characterization-based-on-Fourier-submatrices} provides a proof of Proposition \ref{prop:characterization-based-on-Fourier-submatrices} for self-containedness of the paper.

\section{Preliminaries}
\label{sec:prelim}


\begin{definition}\label{def:HilbertSpace-Sequences}
\rm
A sequence $\{ f_n \}_{n\in\Z}$ in a separable Hilbert space $\mathcal{H}$ is called
\begin{itemize}
\item a \emph{Bessel sequence} in $\mathcal{H}$ (with a Bessel bound $B$) if there is a constant $B>0$ such that
$\sum_{n\in\Z} |\langle f , f_n\rangle|^{2}
\leq B \, \| f \|^2$ for all $f \in \mathcal{H}$;

\item  a \emph{frame} for $\mathcal{H}$ (with frame bounds $A$ and $B$) if there are constants $0 < A \leq B < \infty$ such that
$A \, \| f \|^2
\leq \sum_{n\in\Z} |\langle f , f_n\rangle|^{2}
\leq B \, \| f \|^2$ for all $f \in \mathcal{H}$;

\item  a \emph{Riesz sequence} in $\mathcal{H}$ (with Riesz bounds $A$ and $B$) if there are constants $0 < A \leq B < \infty$ such that
$A \, \| \bc \|_{\ell_2}^2
\leq \| \sum_{n\in\Z} c_n \, f_n \|^2
\leq B \, \| \bc \|_{\ell_2}^2$ for all $\bc = \{c_n\}_{n\in\Z} \in \ell_2 (\mathbb Z)$;

\item  a \emph{Riesz basis} for $\mathcal{H}$ if it is a complete Riesz sequence in $\mathcal{H}$.
\end{itemize}
\end{definition}

It is well-known that a sequence in $\mathcal{H}$ is a Riesz basis if and only if it is both a frame and a Riesz sequence (see e.g., \cite[Proposition 3.7.3, Theorems 5.4.1 and 7.1.1]{Ch16} or \cite[Lemma 1]{KN15}).


Let $\Lambda \subset \R$ be a discrete set and let $S \subset \R$ be a positive measure set.
If $E(\Lambda) := \{e^{2\pi i \lambda (\cdot)} : \lambda\in\Lambda\}$ is a Riesz sequence in $L^2(S)$, then $\Lambda$ is necessarily separated, i.e., $\inf\{ |\lambda - \lambda'| : \lambda \neq \lambda' \in \Lambda \} > 0$ (see e.g., \cite[Proposition 11]{Le21}).
Conversely, if $\Lambda \subset \R$ is separated and $S \subset \R$ is bounded, then $E(\Lambda)$ is a Bessel sequence in $L^2(S)$ (see \cite[p.\,135, Theorem 4]{Yo01}).


\begin{proposition}[see e.g., Lemma 8 in \cite{Le21}]
\label{prop:RB-basic-operations}
Let $\Lambda \subset \R$ be a discrete set and let $S \subset \R$ be a measurable set such that $E(\Lambda)$ is a Riesz basis for $L^2(S)$ with bounds $0 < A \leq B < \infty$.
Then for any $a,b \in \R$ and $c > 0$, the system $E(c \Lambda+a)$ is a Riesz basis for $L^2(\frac{1}{c} S +b)$ with bounds $\frac{A}{c}$ and $\frac{B}{c}$.
Moreover, the system $E(- \Lambda)$ is a Riesz basis for $L^2(- S)$ with bounds $A$ and $B$.
\end{proposition}


\begin{theorem}[Kadec's $1/4$ theorem, see e.g., p.\,36 in \cite{Yo01}]
\label{thm:Kadec}
Let $c > 0$ and $0 < \delta < \frac{1}{4}$.
If a sequence $\Lambda = \{ \lambda_n \}_{n \in \Z} \subset \R$ satisfies $|\lambda_n - c n | \leq c  \delta$ for all $n \in \Z$,
then $E(\Lambda)$ is a Riesz basis for $L^2[0,1/c)$.
\end{theorem}




Fix any $N \in \N$.
For any $\mathcal{K} = \{ c_1, \ldots, c_K \}$ with distinct real numbers $c_1, \ldots , c_K \in [0,N)$, and any $\mathcal{L} \subset \Z_N := \{ 0, 1, \ldots, N{-}1 \}$ with $L := |\mathcal{L}|$, we define
\[
W_{\mathcal{K},\mathcal{L}}
:= \big[
e^{- 2\pi i c_k \ell / N}
\big]_{c_k \in \mathcal{K}, \, \ell \in \mathcal{L}}
\quad \in \C^{ K \times L } .
\]

\begin{proposition}\label{prop:characterization-based-on-Fourier-submatrices}
Fix any $N \in \N$.
Let $\Omega \subset N\Z$ and $S \subset [0,\frac{1}{N})$ be such that $E(\Omega)$ is a Riesz basis for $L^2(S)$ with lower Riesz bound $A > 0$.
Let $\mathcal{K} = \{ c_1, \ldots, c_K \}$ with distinct real numbers $c_1, \ldots , c_K \in [0,N)$, and $\mathcal{L} \subset \Z_N$ with cardinality $L$.
Then $E(\Omega {+} \mathcal{K}) = E(\cup_{k=1}^K \Omega {+} c_k)$ is
\begin{itemize}
\item[(a)] a frame for $L^2(\cup_{\ell \in \mathcal{L}} \, S{+}\frac{\ell}{N})$ if and only if $W_{\mathcal{K},\mathcal{L}}$ is injective;

\item[(b)] a Riesz sequence in $L^2(\cup_{\ell \in \mathcal{L}} \, S{+}\frac{\ell}{N})$ if and only if $W_{\mathcal{K},\mathcal{L}}$ is surjective;

\item[(c)] a Riesz basis for $L^2(\cup_{\ell \in \mathcal{L}} \, S{+}\frac{\ell}{N})$ if and only if $W_{\mathcal{K},\mathcal{L}}$ is bijective.
\end{itemize}
Moreover, the corresponding lower frame/Riesz bound is given by $A \, \sigma_{\min}^2 (W_{\mathcal{K},\mathcal{L}})$, where
$\sigma_{\min} (W_{\mathcal{K},\mathcal{L}})$ is the smallest singular value of $W_{\mathcal{K},\mathcal{L}}$.
\end{proposition}

\begin{remark}
\rm
If $\mathcal{K}$ or $\mathcal{L}$ consists of consecutive integers, then $W_{\mathcal{K},\mathcal{L}}$ is essentially a Vandermonde matrix (due to the multi-linearity of determinant in the rows/columns; see \eqref{eqn:compute-essentially-Vandermonde-determinant} below).
If $N \in \N$ is a prime and $\mathcal{K}, \mathcal{L} \subset \Z_N$, then $W_{\mathcal{K},\mathcal{L}}$ has full rank by Chebotar\"{e}v's theorem on roots of unity, which asserts that every minor of the Fourier matrix in prime dimension is nonzero (see e.g., \cite{SL96}).
\end{remark}

Proposition \ref{prop:characterization-based-on-Fourier-submatrices} is somewhat considered folklore; see for instance, \cite[Theorem 1.1]{D19} and \cite[Lemma 1.13]{OU16-book}.
For self-containedness of the paper, we provide a proof of Proposition \ref{prop:characterization-based-on-Fourier-submatrices} in Section \ref{sec:appendix-proof-prop-characterization-based-on-Fourier-submatrices}. 

Now, let us discuss some examples of exponential Riesz bases.
The first example is based on Proposition \ref{prop:characterization-based-on-Fourier-submatrices}.

\begin{example}\label{ex:RB-N4-periodic}
(a) Let $\mathcal{K}, \mathcal{L} \subset \Z_4$ with $| \mathcal{K} | = | \mathcal{L} |$.
If $\mathcal{K}$ or $\mathcal{L}$ is consecutive,
then $W_{\mathcal{K},\mathcal{L}}$ is essentially a Vandermonde matrix which is invertible, and hence $E(\cup_{k \in \mathcal{K}} \, N\Z {+} k)$ is a Riesz basis for $L^2(\cup_{\ell \in \mathcal{L}} \, [\frac{\ell}{N},\frac{\ell+1}{N}) )$.
On the other hand, for $\mathcal{K} = \mathcal{L} = \{ 0, 2 \}$ the matrix
\[
W_{\mathcal{K},\mathcal{L}} = \big[ e^{- 2\pi i k \ell / 4}
\big]_{k \in \mathcal{K}, \, \ell \in \mathcal{L}}
= \begin{bmatrix}
1 & 1 \\
1 & 1
\end{bmatrix}
\]
is singular, so that
$E(2\Z) = E(4 \Z \cup (4 \Z {+} 2))$ is not a frame nor a Riesz sequence for $L^2([0,\frac{1}{4}) \cup [\frac{2}{4},\frac{3}{4}))$.
In fact, since all functions in $E(2\Z)$ are $\frac{1}{2}$-periodic, the system $E(2\Z)$ is not complete in $L^2([0,\frac{1}{4}) \cup [\frac{2}{4},\frac{3}{4}))$.
However, this is turns out to be a rare case.
If $\mathcal{K}$ is replaced by $\{ 0, c \}$ for any real number $0 < c < 4$, $c \neq 2$, then
\[
W_{\mathcal{K},\mathcal{L}}
= \begin{bmatrix}
1 & 1 \\
1 & e^{- c \pi i}
\end{bmatrix}
\]
is invertible, so that $E(4 \Z \cup (4 \Z {+} c))$ is a Riesz basis for $L^2([0,\frac{1}{4}) \cup [\frac{2}{4},\frac{3}{4}))$.

(b) Let $P \in \N$ be a prime. By Chebotar\"{e}v's theorem on roots of unity (see e.g., \cite{SL96}), the matrix $W_{\mathcal{K},\mathcal{L}}$ is invertible for every $\mathcal{K}, \mathcal{L} \subset \Z_P$ with $| \mathcal{K} | = | \mathcal{L} |$.
Hence, the system $E(\cup_{k \in \mathcal{K}} \, P\Z {+} k)$ is a Riesz basis for $L^2(\cup_{\ell \in \mathcal{L}} \, [\frac{\ell}{P},\frac{\ell+1}{P}) )$ whenever $\mathcal{K}, \mathcal{L} \subset \Z_P$ with $| \mathcal{K} | = | \mathcal{L} |$.
By rescaling (Proposition \ref{prop:RB-basic-operations}), we obtain that
$E(\cup_{k \in \mathcal{K}} \, 4\Z {+} \frac{4}{P}k)$ is a Riesz basis for $L^2(\cup_{\ell \in \mathcal{L}} \, [\frac{\ell}{4},\frac{\ell+1}{4}) )$ whenever $\mathcal{K}, \mathcal{L} \subset \Z_P$ with $| \mathcal{K} | = | \mathcal{L} |$.
\end{example}

The next example is rather pathological and shows the delicate nature of exponential Riesz bases.

\begin{example}\label{ex:TwoRieszBases-1-intersection}
Fix any $0 < \epsilon < \frac{1}{4}$ and set
\[
\begin{split}
\Lambda^{(1)} &:= \{ \ldots , \, -6, \, -4, \, -2, \, 0 , \, 1 {+}\epsilon , \, 3 {+}\epsilon , \, 5 {+}\epsilon , \, \ldots \} , \\
\Lambda^{(2)} &:= - \Lambda^{(1)}
= \{ \ldots , \, -5 {-}\epsilon, \, -3 {-}\epsilon, \, -1 {-}\epsilon, \, 0 , \, 2 , \, 4 , \, 6 , \, \ldots \} .
\end{split}
\]
Then $\Lambda^{(1)} \cap \Lambda^{(2)} = \{ 0 \}$ and
\[
\Lambda^{(1)} \cup \Lambda^{(2)}
= \{ \ldots , \, -3 {-}\epsilon, \, -2, \, -1 {-} \epsilon, \, 0 , \, 1 {+}\epsilon , \, 2 , \, 3 {+}\epsilon , \, \ldots \}  .
\]
By labeling the elements of $\Lambda^{(1)}$ by
\[
\ldots , \;\; \lambda_{-2} = -4, \;\; \lambda_{-1} = -2, \;\; \lambda_0 = 0 , \;\; \lambda_1 = 1 {+}\epsilon , \;\; \lambda_2 = 3 {+} \epsilon , \;\; \ldots ,
\]
and comparing them with elements of $2 \Z {-} \frac{1-\epsilon}{2}$, we see that
\[
\lambda_n - (2n {-} \tfrac{1-\epsilon}{2})
=
\begin{cases}
\frac{1-\epsilon}{2} & \text{for} \;\; n \leq 0 \, ,  \\
-\frac{1-\epsilon}{2} & \text{for} \;\; n \geq 1 \, .
\end{cases}
\]
Then Theorem \ref{thm:Kadec} with $c=2$ and $\delta = \frac{1-\epsilon}{4}$ implies that $E(\Lambda^{(1)})$ is a Riesz basis for $L^2[0,\frac{1}{2})$ and thus, $E(\Lambda^{(2)}) = E(-\Lambda^{(1)})$ is a Riesz basis for $L^2[-\frac{1}{2},0) \cong L^2[\frac{1}{2},1)$.
Also, Theorem \ref{thm:Kadec} with $c=1$ implies that $E(\Lambda^{(1)}\cup \Lambda^{(2)} )$ is a Riesz basis for $L^2[0,1)$.
Consequently, $E(\Lambda^{(1)})$, $E(\Lambda^{(2)})$ and $E(\Lambda^{(1)}\cup \Lambda^{(2)} )$ form Riesz bases for $L^2[0,\frac{1}{2})$, $L^2[\frac{1}{2},1)$ and $L^2[0,1)$ respectively, despite the fact that $\Lambda^{(1)}$ and $\Lambda^{(2)}$ have a nonempty intersection, namely $\Lambda^{(1)} \cap \Lambda^{(2)} = \{ 0 \}$.

Now, let us consider a uniform shifting of $\Lambda^{(2)}$.
For any $0 < \delta < \epsilon \; (< \frac{1}{4})$, we have
$\Lambda^{(1)} \cap (\Lambda^{(2)} {+} \delta) = \emptyset$, $\dist ( \Lambda^{(1)}, \Lambda^{(2)} {+} \delta) = \delta$, and
\[
\Lambda^{(1)} \cup (\Lambda^{(2)} {+} \delta)
= \{ \ldots , \,  -3 {-}\epsilon{+}\delta, \, -2, \, -1 {-} \epsilon{+}\delta, \, 0 , \, \delta , \, 1 {+}\epsilon , \, 2{+}\delta , \, 3 {+}\epsilon , \, \ldots \}  .
\]
Note that removing the element $0$ from $\Lambda^{(1)} \cup (\Lambda^{(2)} {+} \delta)$ yields a pointwise perturbation of $\Z$ by at most $\epsilon \; (< \frac{1}{4})$, so the system $E( (\Lambda^{(1)} \backslash \{ 0 \}) \cup (\Lambda^{(2)} {+} \delta) )$ is a Riesz basis for $L^2[0,1)$ by Theorem \ref{thm:Kadec}.
This implies that $E(\Lambda^{(1)} \cup (\Lambda^{(2)} {+} \delta))$ is overcomplete and thus cannot be a Riesz basis for $L^2[0,1)$, whereas $E(\Lambda^{(1)})$ and $E(\Lambda^{(2)} {+} \delta)$ form Riesz bases for $L^2[0,\frac{1}{2})$ and $L^2[\frac{1}{2},1)$ respectively. 
\end{example}

\section{Comparison with Related Work}
\label{sec:comparison-existing-work}

As mentioned in the introduction, Kozma and Nitzan \cite{KN15} proved that for any finite union $S$ of disjoint intervals in $[0,1)$,
there exists a set $\Lambda \subset \Z$ such that $E(\Lambda)$ is a Riesz basis for $L^2(S)$.
Their result is based on the following lemma which \emph{combines} exponential Riesz bases in a certain way.

For any $N \in \N$, any measurable set $S \subset [0,1)$, and $n = 1, \ldots, N$, let
\[
A_{\geq n}
:= \Big\{ t \in [0, \tfrac{1}{N}) : t+ \tfrac{k}{N} \in S \;\; \text{for at least $n$ values of} \;\; k \in \{ 0 , 1, \ldots, N{-}1 \} \Big\} .
\]

\begin{lemma}[Lemma 2 in \cite{KN15}]
\label{lem:KN15-Lemma2}
Let $N \in \N$ and let $S \subset [0,1)$ be a measurable set.
If there exist sets $\Lambda_1, \ldots, \Lambda_N \subset N\Z$ such that $E(\Lambda_n)$ is a Riesz basis for $L^2(A_{\geq n})$, then $E( \cup_{n=1}^N (\Lambda_n {+} n) )$ is a Riesz basis for $L^2 (S)$.
\end{lemma}

Note that the sets $\Lambda_n$ are shifted by $n$ before taking their union, which prevents overlapping of frequency sets.
When $N$ is prime, one can use \emph{arbitrary} shift factors instead of $n$.

\begin{lemma}[Lemma 6 in \cite{CL22}]
\label{lem:KN15-Lemma2-variant-Nprime}
Let $N \in \N$ be a prime and let $S \subset [0,1)$ be a measurable set.
If there exist sets $\Lambda_1, \ldots, \Lambda_N \subset N\Z$ such that $E(\Lambda_n)$ is a Riesz basis for $L^2(A_{\geq n})$, then for every permutation $\{ k_n \}_{n=1}^N$ of $\{ 1, \ldots, N \}$ the system $E( \cup_{n=1}^N (\Lambda_n {+} k_n) )$ is a Riesz basis for $L^2 (S)$.
\end{lemma}

Using Lemmas \ref{lem:KN15-Lemma2} and \ref{lem:KN15-Lemma2-variant-Nprime}, we obtain the following proposition
which is in the same spirit as our main results, that is, one directly takes unions of frequency sets and domains respectively with appropriate shifts.

\begin{proposition}\label{prop:combineRB-multiple-sets}
Let $S_M  \subset \cdots \subset S_1 \subset [0,\frac{1}{N})$ be measurable sets, where $M,N \in \N$ with $M \leq N$.
Assume that there exist sets $\Lambda_1, \ldots, \Lambda_M \subset N \Z$ such that for each $m = 1,\ldots,M$, the system $E(\Lambda_{m})$ is a Riesz basis for $L^2(S_{m})$. \\
(a) For any permutation $\{ \ell_n \}_{n=1}^N$ of $\{ 1, \ldots, N \}$,
\[
\text{
$E ( \cup_{m=1}^M (\Lambda_{m} {+} m ) )$ is a Riesz basis for $L^2 ( \cup_{m=1}^M (S_{m} {+} \tfrac{\ell_m {-} 1}{N} ) )$.
}
\]
(b) If $N$ is prime, then for any permutations $\{ k_n \}_{n=1}^N$ and $\{ \ell_n \}_{n=1}^N$ of $\{ 1, \ldots, N \}$,
\begin{equation}\label{eqn:combineRB-multiple-sets-PRIME}
\text{
$E ( \cup_{m=1}^M (\Lambda_{m} {+} k_m ) )$ is a Riesz basis for $L^2 ( \cup_{m=1}^M (S_{m} {+} \tfrac{\ell_m {-} 1}{N} ) )$,
}
\end{equation}
and moreover, for any subset $J \subseteq \{ 1,\ldots,M\}$,
\begin{equation}\label{eqn:combineRB-multiple-sets-PRIME-hierarchy}
\text{
$E ( \cup_{m \in J} (\Lambda_{m} {+} k_m ) )$ is a Riesz basis for $L^2 ( \cup_{m \in J} (S_{m} {+} \tfrac{\ell_m {-} 1}{N} ) )$.
}
\end{equation}
\end{proposition}

\medskip

\noindent
\textbf{Comparison with Proposition \ref{prop:characterization-based-on-Fourier-submatrices}}. \
When $a = \frac{1}{N}$, $S = S_1 = \ldots = S_M \subset [0,\frac{1}{N})$ and $\Omega = \Lambda_1 = \ldots = \Lambda_M \subset N\Z$, Proposition \ref{prop:combineRB-multiple-sets} can be deduced from Proposition \ref{prop:characterization-based-on-Fourier-submatrices}.
Indeed, the corresponding matrix $W_{\mathcal{K},\mathcal{L}}$ in part (a) is a Vandermonde matrix, while that in part (b) is a Fourier submatrix which is invertible by Chebotar\"{e}v's theorem on roots of unity.

\medskip


\noindent
\textbf{Comparison with Theorem \ref{thm:two-sets-nested-translation-by-a}}. \
Let $N \geq 2$.
If $S_1  \subset S_2 \subset [0,\frac{1}{N})$ and $\Lambda_1, \Lambda_2 \subset N\Z$ are such that $E(\Lambda_1)$ and $E(\Lambda_2)$ form Riesz bases for $L^2(S_1)$ and $L^2(S_2)$ respectively, then Proposition \ref{prop:combineRB-multiple-sets}(a) implies that
$E ( (\Lambda_1 {+} 1) \cup \Lambda_2 )$ is a Riesz basis for $L^2 ( S_1 \cup (S_2 {+} \frac{K}{N} ) )$ with any $K \in \{ 1, \ldots, N{-}1 \}$.
This result can be also deduced from Theorem \ref{thm:two-sets-nested-translation-by-a}.
Indeed, by setting $S_1' := S_1$, $S_2' := S_2 {+} \frac{K}{N}$, $\Lambda_1' := \Lambda_1 {+} 1$ and $\Lambda_2' := \Lambda_2$, we have $S_1' {+} \frac{K}{N} \subset S_2'$, $\Lambda_1' \subset N\Z{+}1 \, (\subset \R \backslash \frac{N}{K}\Z)$ and $\Lambda_2' \subset N\Z \, (\subset \frac{N}{K}\Z)$, hence, Theorem \ref{thm:two-sets-nested-translation-by-a} with $a = \frac{K}{N}$ yields the same result.
Note that Theorem \ref{thm:two-sets-nested-translation-by-a} actually allows more flexibility in the choice of $\Lambda_1' \subset \R \backslash \frac{N}{K}\Z$.

\begin{proof}[Proof of Proposition \ref{prop:combineRB-multiple-sets}]
(a) For any permutation $\{ \ell_n \}_{n=1}^N$ of $\{ 1, \ldots, N \}$, the set
$S := \bigcup_{m=1}^M ( S_{m} {+} \frac{\ell_m - 1}{N})$
is contained in $[0,1)$, and
\[
A_{\geq n}
= A_{\geq n} (N,S)
=
\begin{cases}
S_n & \text{for} \;\; n = 1, \ldots, M , \\
\emptyset & \text{for} \;\; n = M{+}1, \ldots, N .
\end{cases}
\]
By Lemma \ref{lem:KN15-Lemma2}, the system $E( \cup_{m=1}^M (\Lambda_{m} {+} m) )$ is a Riesz basis for $L^2 (S)$. \\
(b) If $N \in \N$ is prime, then Lemma \ref{lem:KN15-Lemma2-variant-Nprime} yields that for every permutation $\{ k_n \}_{n=1}^N$ of $\{ 1, \ldots, N \}$ the system $E( \cup_{m=1}^M (\Lambda_{m} {+} k_m ))$ is a Riesz basis for $L^2 (S)$.
Now, fix any subset $J \subseteq \{ 1,\ldots,M\}$, say, $J = \{ j_1, \ldots, j_R \}$ with $j_1 < \cdots < j_R$ and $1 \leq R \leq M$.
Since $S_{j_R}  \subset \cdots \subset S_{j_1} \subset [0,\frac{1}{N})$, we can deduce from \eqref{eqn:combineRB-multiple-sets-PRIME} that
$E( \cup_{r=1}^R (\Lambda_{j_r} {+} k_{j_r} ) )$ is a Riesz basis for $L^2 (\cup_{r=1}^R (S_{j_r} {+}  \frac{\ell_{j_r} {-} 1}{N} ))$, which is exactly \eqref{eqn:combineRB-multiple-sets-PRIME-hierarchy}.
\end{proof}

\section{Proof of Theorem \ref{thm:two-sets-nested-translation-by-a}}
\label{sec:Proof-of-thm:two-sets-nested-translation-by-a}

In the proof of Theorems \ref{thm:two-sets-nested-translation-by-a}, \ref{thm:RB-union-two-sets-generalization}, \ref{thm:RB-union-two-sets-generalization-converse} and \ref{thm:two-sets-nested-translation-by-a-dimension-d}, we will need the following notions.
The \emph{Fourier transform} is defined as 
\[
\mathcal{F} [f] (\omega) = \widehat{f} (\omega) = \int_{-\infty}^{\infty} f(t) \, e^{- 2 \pi i t \omega} \, dt , 
\quad f \in L^1(\R) \cap L^2(\R) , 
\] 
so that $\mathcal{F}[\cdot]$ extends to a unitary operator from
$L^{2}(\mathbb{R})$ onto $L^{2}(\mathbb{R})$.
The \emph{Paley-Wiener space} over a measurable set $S \subset \R$ is defined as
\[
PW(S) := \{ f \in L^2(\R) : \supp \widehat{f} \subset S \}
\;=\; \mathcal{F}^{-1} \big[ L^2(S) \big]
\]
equipped with the norm $\| f \|_{PW(S)} := \| f \|_{L^2(\R)} = \| \widehat{f} \|_{L^2(S)}$. 
If $S \subset \R$ has finite measure, then every $f \in PW(S)$ is continuous and
\begin{equation}\label{eqn:relation-between-PW-and-L2}
f(x) = \int_{S} \widehat{f} (\omega) \, e^{2 \pi i x \omega} \, d\omega
= \big\langle \widehat{f} , e^{- 2 \pi i x (\cdot)} \big\rangle_{L^2(S)}
\quad \text{for all} \;\; x \in \R .
\end{equation}

\begin{definition}
\label{def:set-uniqueness-interpolation}
\rm
Let $S \subset \R$ be a measurable set.
A discrete set $\Lambda \subset \R$ is called
\begin{itemize}
\item
\emph{a set of uniqueness} for $PW(S)$ if the only function $f \in  PW(S)$ satisfying $f(\lambda) = 0$ for all $\lambda \in \Lambda$ is the trivial function $f=0$;


\item
\emph{a set of interpolation} for $PW(S)$ if for each $\{ c_ \lambda \}_{\lambda \in \Lambda} \in \ell_2(\Lambda)$ there exists a function $f \in  PW(S)$ satisfying $f(\lambda) = c_\lambda$ for all $\lambda \in \Lambda$.
\end{itemize}
\end{definition}

The relation \eqref{eqn:relation-between-PW-and-L2} easily implies the following.
\begin{itemize}
\item[({\romannumeral 1})]
$\Lambda$ is a set of uniqueness for $PW(S)$ if and only if $E(-\Lambda) = \overline{E(\Lambda)}$ is complete in $L^2(S)$ if and only if $E(\Lambda)$ is complete in $L^2(S)$.


\item[({\romannumeral 2})]
If $E(\Lambda)$ is a Bessel sequence in $L^2(S)$, then $\Lambda$ is a set of interpolation for $PW(S)$ if and only if $E(\Lambda)$ is a Riesz sequence in $L^2(S)$ (see \cite[Appendix A]{Le21} for more details).
\end{itemize}

We are now ready to prove Theorem \ref{thm:two-sets-nested-translation-by-a}.

\medskip

\noindent
\textbf{Completeness}. \
Assume that $f \in PW (S_1 \cup S_2)$ satisfies $f(\lambda) = 0$
for $\lambda \in \Lambda_1 \cup \Lambda_2$.
Define $f_1\in PW(S_1)$ by $\widehat{f}_1=\widehat{f} \cdot \mathds{1}_{S_1}$ and define $g\in PW(S_2)$ by
\[
 \widehat g(\omega)
 :=
 \underbrace{\widehat f(\omega)-\widehat{f}_1(\omega)}_{\text{supported in} \; S_2 }
 + \underbrace{\widehat{f}_1(\omega - a)}_{\text{supported in} \; S_1 {+} a \, \subset S_2 } .
\]
Since $S_2$ is a bounded set, it follows from \eqref{eqn:relation-between-PW-and-L2} that $g(x)=  f(x) - (1- e^{2\pi i a x}) \, f_1(x)$ for all $x \in \R$; in particular, we have $g(\lambda_2)= f(\lambda_2)=0$ for all $\lambda_2 \in \Lambda_2 \; (\subset \frac{1}{a}\Z)$. Since $g\in PW(S_2)$ and since $\Lambda_2$ is a set of uniqueness for $PW(S_2)$, we have $g=0$ and thus $f(x) = (1-e^{2\pi i a x}) \, f_1(x)$ for all $x \in \R$.
By the assumption, it holds $0 = f(\lambda_1)= (1-e^{2\pi i a \lambda_1}) \, f_1(\lambda_1)$ for all $\lambda_1 \in \Lambda_1$.
Note that for every $\lambda_1 \in \Lambda_1 \; (\subset \R \backslash \frac{1}{a}\Z)$, we have $e^{2\pi i a \lambda_1} \neq 1$ and therefore $f_1(\lambda_1) = 0$.
Since $f_1 \in PW(S_1)$ and since $\Lambda_1$ is a set of uniqueness for $PW(S_1)$, we get $f_1 = 0$ and thus arrive at $f=0$.
This shows that $\Lambda_1 \cup \Lambda_2$ is a set of uniqueness for $PW (S_1 \cup S_2)$,
hence, the system $E(\Lambda_1 \cup \Lambda_2)$ is complete in $L^2 (S_1 \cup S_2)$.

\medskip

\noindent
\textbf{Riesz sequence}. \
Since $\Lambda_1 \cup \Lambda_2$ is separated and since $S_1 \cup S_2$ is bounded, the system $E (\Lambda_1 \cup \Lambda_2)$ is a Bessel sequence in $L^2 (S_1 \cup S_2)$. Thus, to prove that $E (\Lambda_1 \cup \Lambda_2)$ is a Riesz sequence in $L^2 (S_1 \cup S_2)$, it is enough to show that $\Lambda_1 \cup \Lambda_2$ is a set of interpolation for $PW (S_1 \cup S_2)$.

To show this, we will fix an arbitrary $\boldb =\{ b_{\lambda} \}_{\lambda \in \Lambda_1 \cup \Lambda_2} \in \ell_2 (\Lambda_1 \cup \Lambda_2)$ and construct a function $f \in PW (S_1 \cup S_2)$ satisfying $f(\lambda) = b_{\lambda}$ for all $\lambda \in \Lambda_1 \cup \Lambda_2$.
Since $\Lambda_2$ is a set of interpolation for $PW(S_2)$, there exists a function $g \in PW(S_2)$ satisfying $g(\lambda_2)=b_{\lambda_2}$ for all $\lambda_2 \in \Lambda_2$.
Note that for any $f_1 \in PW(S_1)$, the function
\[
f(x) :=  g(x) + (1-e^{2\pi i a x}) \, f_1(x) ,
\quad x \in \R ,
\]
belongs in $PW (S_1 \cup S_2)$ and
satisfies $f(\lambda_2)= g(\lambda_2) = b_{\lambda_2}$ for $\lambda_2 \in \Lambda_2 \; (\subset \frac{1}{a}\Z)$.
We will now choose a particular function $f_1 \in PW(S_1)$ as follows.
Since $\Lambda_1$ is separated and since $S_2$ is bounded, the system $E(\Lambda_1)$ is a Bessel sequence in $L^2(S_2)$, so that the sequence $\{ g(\lambda_1) \}_{\lambda_1 \in \Lambda_1}$ is in $\ell_2(\Lambda_1)$ and so is the sequence $\big\{ \frac{1}{ 1-e^{2\pi i a \lambda_1} } \big( b_{\lambda_1}-g(\lambda_1) \big) \big\}_{\lambda_1 \in \Lambda_1}$.
Here, it holds $e^{2\pi i a \lambda_1} \neq 1$ for all $\lambda_1 \in \Lambda_1 \; (\subset \R \backslash \frac{1}{a}\Z)$, and moreover the condition $\dist (\Lambda_1, \frac{1}{a}\Z) > 0$ ensures that
$\inf_{\lambda_1 \in \Lambda_1} | 1-e^{2\pi i a \lambda_1} | > 0$.
Now, since $\Lambda_1$ is a set of interpolation for $PW(S_1)$, there exists a function $f_1 \in PW(S_1)$ satisfying $f_1(\lambda_1)= \frac{1}{ 1-e^{2\pi i a \lambda_1} } \big( b_{\lambda_1}-g(\lambda_1) \big)$ for all $\lambda_1 \in \Lambda_1$.
We then have $f(\lambda_1)=b_{\lambda_1}$ for all $\lambda_1 \in \Lambda_1$. This shows that the function $f \in PW (S_1 \cup S_2)$ satisfies $f(\lambda) = b_{\lambda}$ for all $\lambda \in \Lambda_1 \cup \Lambda_2$, as desired.

\medskip

\noindent
\textbf{Riesz basis}. \
By definition, a Riesz basis is a complete Riesz sequence. Hence, this part follows immediately by combining the \emph{completeness} and \emph{Riesz sequence} parts.
\hfill $\Box$ 

\begin{remark}
\rm
Fix any nonzero real number $a \neq 0$.
Then for $x \in \R$, the term $1-e^{2\pi i a x}$ vanishes if and only if $a x \in \Z$.
Similarly, for any $M \in \N$ and $x \in \R$, the term
\begin{equation}\label{eqn:one-minus-exp-factorization}
\begin{split}
1-e^{2\pi i M a x}
&= (1-e^{2\pi i a x}) \cdot ( 1+ e^{2\pi i a x} + \ldots + e^{2\pi i (M-1) a x} )  \\
\end{split}
\end{equation}
vanishes if and only if $e^{2\pi i a x}$ is an $M$-th root of unity (i.e., $e^{2\pi i a x}$ is one of $1, \, e^{2\pi i /M}, \, e^{2\pi i 2/M}, \, \ldots, \, e^{2\pi i (M-1)/M}$)
if and only if $a x \in \frac{1}{M} \Z$.
We will make use of the factorization \eqref{eqn:one-minus-exp-factorization} in the proof of Theorem \ref{thm:RB-union-two-sets-generalization}.
\end{remark}


\section{Proof of Theorem \ref{thm:RB-union-two-sets-generalization}}
\label{sec:Proof-of-thm:RB-union-two-sets-generalization}

By shifting the sets $\Lambda_1, \Lambda_2 \subset \R$ simultaneously by a constant (see Proposition \ref{prop:RB-basic-operations}), we may assume without loss of generality that $c_1 = 0$.

\medskip

\noindent
\textbf{Completeness}. \
Assume that $f \in PW (S_1 \cup S_2)$ satisfies $f(\lambda) = 0$
for $\lambda \in \Lambda_1 \cup \Lambda_2$.
Define $f_1\in PW(S_1)$ by $\widehat{f}_1=\widehat{f} \cdot \mathds{1}_{S_1}$ and define $g\in PW(S_2)$ by
\[
 \widehat g(\omega)
 :=
 \underbrace{\widehat f(\omega)-\widehat{f}_1(\omega)}_{\text{supported in} \; S_2 }
 + \underbrace{a_1 \, \widehat{f}_1(\omega - a)}_{\text{supported in} \; S_1 {+} a \, \subset S_2 }
 + \;\cdots\;
 + \underbrace{a_K \, \widehat{f}_1(\omega - Ka)}_{\text{supported in} \; S_1 {+} Ka \, \subset S_2 }
\]
for some constants $a_1, \ldots, a_K \in \R$ to be chosen later.
Requiring $a_1 + \ldots + a_K = 1$, one may rewrite the equation as
\[
 \widehat g(\omega)
 =
 \widehat f(\omega)
 - a_1 \, \big(\widehat{f}_1(\omega) - \widehat{f}_1(\omega - a) \big)
 - \;\cdots\;
 - a_K \, \big(\widehat{f}_1(\omega) - \widehat{f}_1(\omega - Ka) \big),
\]
which is equivalent to
\[
g(x)=  f(x) - a_1 \, (1- e^{2\pi i a x}) \, f_1(x)   - \;\cdots\;  - a_K \, (1- e^{2\pi i K a x}) \, f_1(x) .
\]
Using the factorization \eqref{eqn:one-minus-exp-factorization}, this is in turn equivalent to
\begin{equation}\label{eqn:gx-using-hx}
g(x) =  f(x) \, - \, f_1(x) \cdot (1- e^{2\pi i a x}) \cdot h(x)
\end{equation}
where $h(x) := a_1   + a_2 \, (1+ e^{2\pi i a x}) + \cdots  + a_K \, (1+ \cdots + e^{2\pi i (K-1) a x})$.
Changing the variables from $a_1, \ldots, a_K$ to $a_k' := \sum_{n=k}^K a_n$ for $k =1,\ldots,K$ (which corresponds to a linear bijection between $(a_1, \ldots, a_K) \in \R^K$ and $(a_1', \ldots, a_K') \in \R^K$),
the constraint $a_1 + \ldots + a_K = 1$ becomes $a_1' = 1$. Thus, we have
\begin{equation}\label{eqn:hx-after-change-of-variables}
h(x) = 1  + a_2' \, e^{2\pi i a x} + \cdots  + a_K' \, e^{2\pi i (K-1) a x} .
\end{equation}
We will now find $a_2', \ldots, a_K' \in \R$ satisfying $h (c_k) = 0$ for $k=2, \ldots, K$,
which corresponds to solving the $(K{-}1){\times} (K{-}1)$ linear system
\[
\begin{bmatrix}
e^{2\pi i a c_2} & e^{2\pi i 2 a c_2} & \cdots & e^{2\pi i (K-1) a c_2} \\
e^{2\pi i a c_3} & e^{2\pi i 2 a c_3} & \cdots & e^{2\pi i (K-1) a c_3} \\
\vdots & \vdots & \ddots & \vdots \\
e^{2\pi i a c_K} & e^{2\pi i 2 a c_K} & \cdots & e^{2\pi i (K-1) a c_K}
\end{bmatrix}
\begin{bmatrix}
a_2' \\
a_3' \\
\vdots \\
a_K'
\end{bmatrix}
=
\begin{bmatrix}
-1 \\
-1 \\
\vdots \\
-1
\end{bmatrix} .
\]
Note that the associated matrix has a \emph{nonzero} determinant given by
\begin{equation}\label{eqn:compute-essentially-Vandermonde-determinant}
\begin{split}
& e^{2\pi i a (c_2+c_3+\ldots+c_K)}
\cdot
\det
\begin{bmatrix}
1 & e^{2\pi i a c_2} & \cdots & e^{2\pi i (K-2) a c_2} \\
1 & e^{2\pi i a c_3} & \cdots & e^{2\pi i (K-2) a c_3} \\
\vdots & \vdots & \ddots & \vdots \\
1 & e^{2\pi i a c_K} & \cdots & e^{2\pi i (K-2) a c_K}
\end{bmatrix}  \\
&= e^{2\pi i a (c_2+c_3+\ldots+c_K) }
\cdot
\prod_{2 \leq j < k \leq K} ( e^{2\pi i a c_k} - e^{2\pi i a c_j} )
\neq 0
\end{split}
\end{equation}
where we used the multi-linearity of determinant in the rows, and the Vandermonde determinant formula (see e.g., \cite{HJ13}).
Hence, the linear system is \emph{uniquely solvable} in $a_2', \ldots, a_K' \in \R$, and in turn, $a_1, \ldots, a_K \in \R$ are \emph{uniquely determined} from $a_2', \ldots, a_K'$ via the 1:1 correspondence.
In fact, we must have
\[
h(x) = \prod_{k=2}^K \, (1- e^{2\pi i a (x-c_k)})
\]
which is clearly of the form \eqref{eqn:hx-after-change-of-variables} and satisfies $h (c_k) = 0$ for $k=2, \ldots, K$.
Since $c_1 = 0$, it follows from \eqref{eqn:gx-using-hx} that
\begin{equation}\label{eqn:gx-solution}
g(x)
=  f(x) \, - \, f_1(x) \cdot  \prod_{k=1}^K (1- e^{2\pi i a (x-c_k)})  .
\end{equation}

Now, observe that $g(\lambda_2)= f(\lambda_2) = 0$ for all $\lambda_2 \in \Lambda_2 \subset \cup_{k=1}^K (\frac{1}{a}\Z{+}c_k)$.
Since $g\in PW(S_2)$ and since $\Lambda_2$ is a set of uniqueness for $PW(S_2)$, we have $g=0$ and thus $f(x) = f_1(x) \cdot \prod_{k=1}^K (1- e^{2\pi i a (x-c_k)})$ for all $x \in \R$.
By the assumption, it holds $0 = f(\lambda)= f_1(\lambda) \cdot \prod_{k=1}^K (1- e^{2\pi i a (\lambda-c_k)})$ for all $\lambda \in \Lambda_1$.
Note that for every $\lambda \in \Lambda_1 \subset \R \backslash \cup_{k=1}^K (\frac{1}{a}\Z{+}c_k)$, we have $\prod_{k=1}^K (1- e^{2\pi i a (\lambda-c_k)}) \neq 0$ and therefore $f_1(\lambda) = 0$.
Since $f_1 \in PW(S_1)$ and since $\Lambda_1$ is a set of uniqueness for $PW(S_1)$, we get $f_1 = 0$ and arrive at $f=0$.
Hence, $\Lambda_1 \cup \Lambda_2$ is a set of uniqueness for $PW (S_1 \cup S_2)$, equivalently, $E(\Lambda_1 \cup \Lambda_2)$ is complete in $L^2 (S_1 \cup S_2)$.

\medskip

\noindent
\textbf{Riesz sequence}. \
Since $\Lambda_1 \cup \Lambda_2$ is separated and since $S_1 \cup S_2$ is bounded, the system $E (\Lambda_1 \cup \Lambda_2)$ is a Bessel sequence in $L^2 (S_1 \cup S_2)$. Thus, to prove that $E (\Lambda_1 \cup \Lambda_2)$ is a Riesz sequence in $L^2 (S_1 \cup S_2)$, it is enough to show that $\Lambda_1 \cup \Lambda_2$ is a set of interpolation for $PW (S_1 \cup S_2)$.

To show this, we will fix an arbitrary $\boldb =\{ b_{\lambda} \}_{\lambda \in \Lambda_1 \cup \Lambda_2} \in \ell_2 (\Lambda_1 \cup \Lambda_2)$ and construct a function $f \in PW (S_1 \cup S_2)$ satisfying $f(\lambda) = b_{\lambda}$ for all $\lambda \in \Lambda_1 \cup \Lambda_2$.
Since $\Lambda_2$ is a set of interpolation for $PW(S_2)$, there exists a function $g \in PW(S_2)$ satisfying $g(\lambda_2)=b_{\lambda_2}$ for all $\lambda_2 \in \Lambda_2$.
Note that for any $f_1 \in PW(S_1)$, the function
\[
f(x) =  g(x) + f_1(x) \cdot  \prod_{k=1}^K (1- e^{2\pi i a (x-c_k)})  ,
\quad x \in \R ,
\]
belongs in $PW (S_1 \cup S_2)$ (this is seen by following the derivations \eqref{eqn:gx-using-hx}-\eqref{eqn:gx-solution} backwards) and
satisfies $f(\lambda_2)= g(\lambda_2) = b_{\lambda_2}$ for $\lambda_2 \in \Lambda_2 \subset \cup_{k=1}^K (\frac{1}{a}\Z{+}c_k)$.
We will now choose a particular function $f_1 \in PW(S_1)$ as follows.
Since $\Lambda_1$ is separated and since $S_2$ is bounded, the system $E(\Lambda_1)$ is a Bessel sequence in $L^2(S_2)$, so that the sequence $\{ g(\lambda_1) \}_{\lambda_1 \in \Lambda_1}$ is in $\ell_2(\Lambda_1)$ and so is the sequence
$\big\{ \frac{1}{ \prod_{k=1}^K (1- e^{2\pi i a (\lambda_1-c_k)}) } \big( b_{\lambda_1}-g(\lambda_1) \big) \big\}_{\lambda_1 \in \Lambda_1}$.
Note that $\prod_{k=1}^K (1- e^{2\pi i a (\lambda_1-c_k)}) \neq 0$ for all $\lambda_1 \in \Lambda_1 \subset \R \backslash \cup_{k=1}^K (\frac{1}{a}\Z{+}c_k)$, and moreover, $\dist (\Lambda_1, \cup_{k=1}^K (\frac{1}{a}\Z{+}c_k)) > 0$ implies that
\[
\inf_{\lambda_1 \in \Lambda_1}  \big| \textstyle \prod_{k=1}^K (1- e^{2\pi i a (\lambda_1-c_k)}) \big| > 0 .
\]
Now, since $\Lambda_1$ is a set of interpolation for $PW(S_1)$, there exists a function $f_1 \in PW(S_1)$ satisfying $f_1(\lambda_1)= \frac{1}{ \prod_{k=1}^K (1- e^{2\pi i a (\lambda_1-c_k)}) } \big( b_{\lambda_1}-g(\lambda_1) \big)$ for all $\lambda_1 \in \Lambda_1$.
We then have $f(\lambda_1)=b_{\lambda_1}$ for all $\lambda_1 \in \Lambda_1$. Hence, the function $f \in PW (S_1 \cup S_2)$ satisfies $f(\lambda) = b_{\lambda}$ for all $\lambda \in \Lambda_1 \cup \Lambda_2$, as desired.

\medskip

\noindent
\textbf{Riesz basis}. \
This part follows immediately by combining the \emph{completeness} and \emph{Riesz sequence} parts.
\hfill $\Box$ 

\section{Proof of Theorem \ref{thm:RB-union-two-sets-generalization-converse}}
\label{sec:Proof-of-prop:RB-union-two-sets-generalization-converse}

\noindent
(a) To prove that $E(\Lambda_1)$ is complete in $L^2(S_1)$, we will assume that $f \in PW (S_1)$ satisfies $f(\lambda) = 0$ for $\lambda \in \Lambda_1$, and show that $f = 0$.
Define $g\in PW( S_1 \cup S_2 )$ by
\[
 \widehat g(\omega)
 :=
 \underbrace{\widehat{f} (\omega)}_{\text{supported in} \; S_1 }
 - \underbrace{\widehat{f}(\omega - a)}_{\text{supported in} \; S_1 {+} a \, \subset S_2 } .
\]
Then $g(x)=  (1- e^{2\pi i a x}) \, f(x)$ for $x \in \R$, so that $g(\lambda) = 0$ for all $\lambda \in \Lambda_1 \cup \Lambda_2$.
Since $g\in PW( S_1 \cup S_2 )$ and since $\Lambda_1 \cup \Lambda_2$ is a set of uniqueness for $PW(S_1 \cup S_2)$, we get $g=0$, so that $f(x)=0$ for all $x \in \R \backslash \frac{1}{a}\Z$.
Since $f \in PW (S_1)$ is analytic, we obtain $f=0$ by the uniqueness of analytic continuation.
Hence, $\Lambda_1$ is a set of uniqueness for $PW (S_1)$, equivalently, $E(\Lambda_1)$ is complete in $L^2 (S_1)$. \\
(b) As in part (a), we will assume that $f \in PW (S_1)$ satisfies $f(\lambda) = 0$ for $\lambda \in \Lambda_1$, and show that $f = 0$.
Define $g\in PW( S_1 \cup S_2 )$ by
\[
 \widehat g(\omega)
 :=
 \underbrace{\widehat{f}_1(\omega)}_{\text{supported in} \; S_1 }
 - \underbrace{a_1 \, \widehat{f}_1(\omega - a)}_{\text{supported in} \; S_1 {+} a \, \subset S_2 }
 - \;\cdots\;
 - \underbrace{a_K \, \widehat{f}_1(\omega - Ka)}_{\text{supported in} \; S_1 {+} Ka \, \subset S_2 }
\]
for some $a_1, \ldots, a_K \in \R$ satisfying $a_1 + \ldots + a_K = 1$,
so that
\[
g(x)=  a_1 \, (1- e^{2\pi i a x}) \, f(x)   + \;\cdots\;  + a_K \, (1- e^{2\pi i K a x}) \, f(x) .
\]
By similar arguments as in the proof of Theorem \ref{thm:RB-union-two-sets-generalization}, the constants $a_1, \ldots, a_K \in \R$ are \emph{uniquely determined} by requiring $g (c_k) = 0$ for $k=2, \ldots, K$. In fact, due to the \emph{uniqueness} we obtain
\[
g(x) =  f(x) \cdot  \prod_{k=1}^K (1- e^{2\pi i a (x-c_k)})
\]
Now, observe that $g(\lambda) = 0$ for all $\lambda \in \Lambda_1 \cup \Lambda_2$.
Since $g\in PW( S_1 \cup S_2 )$ and since $\Lambda_1 \cup \Lambda_2$ is a set of uniqueness for $PW(S_1 \cup S_2)$, we get $g=0$, so that $f(x)=0$ for all $x \in \R \backslash \cup_{k=1}^K (\frac{1}{a}\Z{+}c_k)$.
Since $f \in PW (S_1)$ is analytic, we obtain $f=0$ by the uniqueness of analytic continuation.
Hence, $\Lambda_1$ is a set of uniqueness for $PW (S_1)$, equivalently, $E(\Lambda_1)$ is complete in $L^2 (S_1)$.
\hfill $\Box$ 

\section{Proof of Theorem \ref{thm:two-sets-nested-translation-by-a-dimension-d}}
\label{sec:Proof-of-thm:two-sets-nested-translation-by-a-dimension-d}


We will use similar arguments as in the proof of Theorem \ref{thm:two-sets-nested-translation-by-a}. 
The $d$-dimensional \emph{Fourier transform} is defined as 
\[
\mathcal{F} [f] (\boldsymbol{\omega}) = \widehat{f} (\boldsymbol{\omega}) = \int_{\R^d} f(\mathbf{x} ) \, e^{- 2 \pi i \mathbf{x}  \cdot \boldsymbol{\omega}} \, d\mathbf{x} ,
\quad f \in L^1(\R^d) \cap L^2(\R^d) , 
\] 
so that $\mathcal{F}[\cdot]$ extends to a unitary operator from
$L^{2}(\mathbb{R}^d)$ onto $L^{2}(\mathbb{R}^d)$.
%
%
%
%
%
%
The \emph{Paley-Wiener space} over a measurable set $S \subset \R^d$ is defined as
\[
PW(S) := \{ f \in L^2(\R^d) : \supp \widehat{f} \subset S \}
\;=\; \mathcal{F}^{-1} \big[ L^2(S) \big]
\]
equipped with the norm $\| f \|_{PW(S)} := \| f \|_{L^2(\R^d)} = \| \widehat{f} \|_{L^2(S)}$. 
If $S \subset \R^d$ has finite measure, then every $f \in PW(S)$ is continuous and
\[
f(\mathbf{x} ) = \int_{S} \widehat{f} (\boldsymbol{\omega}) \, e^{- 2 \pi i \mathbf{x}  \cdot \boldsymbol{\omega}} \, d\boldsymbol{\omega}
\quad \text{for all} \;\; \mathbf{x}  \in \R^d .
\]
The notion of \emph{a set of uniqueness/interpolation} for $PW(S)$ with $S \subset \R^d$ is defined similarly as in Definition \ref{def:set-uniqueness-interpolation}. 

\medskip

\noindent
\textbf{Completeness}. \
Assume that $f \in PW (S_1 \cup S_2)$ satisfies $f(\lambda) = 0$
for $\lambda \in \Lambda_1 \cup \Lambda_2$.
Define $f_1\in PW(S_1)$ by $\widehat{f}_1=\widehat{f} \cdot \mathds{1}_{S_1}$ and define $g\in PW(S_2)$ by
\[
 \widehat g(\boldsymbol{\omega})
 :=
 \underbrace{\widehat f(\boldsymbol{\omega})-\widehat{f}_1(\boldsymbol{\omega})}_{\text{supported in} \; S_2 }
 + \underbrace{\widehat{f}_1\big( \boldsymbol{\omega} - (a_1, \ldots, a_d) \big)}_{\text{supported in} \; S_1 {+} (a_1, \ldots, a_d)  \, \subset S_2 } .
\]
Since $S_2$ is a bounded set, it follows that $g$ is continuous and 
$g(x_1, \ldots, x_d) =  f(x_1, \ldots, x_d)  - (1- e^{2\pi i (a_1 x_1 + \cdots + a_d x_d )} ) \, f_1(x_1, \ldots, x_d)$ for all $(x_1, \ldots, x_d) \in \R^d$; in particular, we have $g(\lambda_2)= f(\lambda_2)=0$ for all $\lambda_2 \in \Lambda_2 \; (\subset \Gamma_{a_1, \ldots, a_d } )$. 
Since $g\in PW(S_2)$ and since $\Lambda_2$ is a set of uniqueness for $PW(S_2)$, we have $g=0$ and thus $f(x_1, \ldots, x_d)  = (1- e^{2\pi i (a_1 x_1 + \cdots + a_d x_d )} ) \, f_1(x_1, \ldots, x_d) $ for all $(x_1, \ldots, x_d)  \in \R^d$.
By the assumption, it holds $0 = f(\lambda_1)= (1-e^{2\pi i (a_1, \ldots, a_d) \cdot \lambda_1}) \, f_1(\lambda_1)$ for all $\lambda_1 \in \Lambda_1$.
Note that for every $\lambda_1 \in \Lambda_1 \; (\subset \R^d \backslash \Gamma_{a_1, \ldots, a_d })$, we have $e^{2\pi i (a_1, \ldots, a_d) \cdot \lambda_1} \neq 1$ and therefore $f_1(\lambda_1) = 0$.
Since $f_1 \in PW(S_1)$ and since $\Lambda_1$ is a set of uniqueness for $PW(S_1)$, we get $f_1 = 0$ and thus arrive at $f=0$.
This shows that $\Lambda_1 \cup \Lambda_2$ is a set of uniqueness for $PW (S_1 \cup S_2)$,
hence, the system $E(\Lambda_1 \cup \Lambda_2)$ is complete in $L^2 (S_1 \cup S_2)$.

\medskip

\noindent
\textbf{Riesz sequence}. \
Since $\Lambda_1 \cup \Lambda_2$ is separated and since $S_1 \cup S_2$ is bounded, the system $E (\Lambda_1 \cup \Lambda_2)$ is a Bessel sequence in $L^2 (S_1 \cup S_2)$. Thus, to prove that $E (\Lambda_1 \cup \Lambda_2)$ is a Riesz sequence in $L^2 (S_1 \cup S_2)$, it is enough to show that $\Lambda_1 \cup \Lambda_2$ is a set of interpolation for $PW (S_1 \cup S_2)$.

To show this, we will fix an arbitrary $\boldb =\{ b_{\lambda} \}_{\lambda \in \Lambda_1 \cup \Lambda_2} \in \ell_2 (\Lambda_1 \cup \Lambda_2)$ and construct a function $f \in PW (S_1 \cup S_2)$ satisfying $f(\lambda) = b_{\lambda}$ for all $\lambda \in \Lambda_1 \cup \Lambda_2$.
Since $\Lambda_2$ is a set of interpolation for $PW(S_2)$, there exists a function $g \in PW(S_2)$ satisfying $g(\lambda_2)=b_{\lambda_2}$ for all $\lambda_2 \in \Lambda_2$.
Note that for any $f_1 \in PW(S_1)$, the function
\[
f(x_1, \ldots, x_d)  :=  g(x_1, \ldots, x_d)  + (1-  e^{2\pi i (a_1 x_1 + \cdots + a_d x_d )} ) \, f_1(x_1, \ldots, x_d)  
\]
belongs in $PW (S_1 \cup S_2)$ and
satisfies $f(\lambda_2)= g(\lambda_2) = b_{\lambda_2}$ for $\lambda_2 \in \Lambda_2 \; (\subset \Gamma_{a_1, \ldots, a_d } )$.
We will now choose a particular function $f_1 \in PW(S_1)$ as follows.
Since $\Lambda_1$ is separated and since $S_2$ is bounded, the system $E(\Lambda_1)$ is a Bessel sequence in $L^2(S_2)$, so that the sequence $\{ g(\lambda_1) \}_{\lambda_1 \in \Lambda_1}$ is in $\ell_2(\Lambda_1)$ and so is the sequence $\big\{ \frac{1}{ 1-e^{2\pi i (a_1, \ldots, a_d) \cdot \lambda_1} } \big( b_{\lambda_1}-g(\lambda_1) \big) \big\}_{\lambda_1 \in \Lambda_1}$.
Here, it holds $e^{2\pi i (a_1, \ldots, a_d) \cdot \lambda_1} \neq 1$ for all $\lambda_1 \in \Lambda_1 \; (\subset \R^d \backslash \Gamma_{a_1, \ldots, a_d })$, and moreover the condition $\dist (\Lambda_1, \Gamma_{a_1, \ldots, a_d }) > 0$ ensures that
$\inf_{\lambda_1 \in \Lambda_1} | 1-e^{2\pi i (a_1, \ldots, a_d) \cdot \lambda_1} | > 0$.
Now, since $\Lambda_1$ is a set of interpolation for $PW(S_1)$, there exists a function $f_1 \in PW(S_1)$ satisfying $f_1(\lambda_1)= \frac{1}{ 1-e^{2\pi i (a_1, \ldots, a_d) \cdot \lambda_1} } \big( b_{\lambda_1}-g(\lambda_1) \big)$ for all $\lambda_1 \in \Lambda_1$.
We then have $f(\lambda_1)=b_{\lambda_1}$ for all $\lambda_1 \in \Lambda_1$. This shows that the function $f \in PW (S_1 \cup S_2)$ satisfies $f(\lambda) = b_{\lambda}$ for all $\lambda \in \Lambda_1 \cup \Lambda_2$, as desired.

\medskip

\noindent
\textbf{Riesz basis}. \
By definition, a Riesz basis is a complete Riesz sequence. Hence, this part follows immediately by combining the \emph{completeness} and \emph{Riesz sequence} parts.
\hfill $\Box$ 

\section{Proof of Proposition \ref{prop:characterization-based-on-Fourier-submatrices}}
\label{sec:appendix-proof-prop-characterization-based-on-Fourier-submatrices}

\noindent
\textbf{(a) Frame}. \
Fix any $f \in L^2(\cup_{\ell \in \mathcal{L}} \, S{+}\frac{\ell}{N})$. Then for any $\lambda \in \Omega {+} c_k \; (\subset N\Z {+} c_k)$ with $1 \leq k \leq K$,
\[
\begin{split}
\langle f, e^{2\pi i \lambda (\cdot)} \rangle_{L^2(\cup_{\ell \in \mathcal{L}} \, S{+}\frac{\ell}{N})}
&= \int_{\cup_{\ell \in \mathcal{L}} \, S{+}\frac{\ell}{N}} f(t) \, e^{- 2\pi i \lambda t} \, dt \\
&= \int_S \sum_{\ell \in \mathcal{L}} f ( t {+} \tfrac{\ell}{N} ) \, e^{- 2\pi i \lambda (t + \frac{\ell}{N}) } \, dt \\
&= \int_S F_{k}(t) \, e^{- 2\pi i \lambda t} \, dt = \langle F_{k}, e^{2\pi i \lambda (\cdot)} \rangle_{L^2(S)}
\end{split}
\]
where
\begin{equation}\label{eqn:def-Fm}
F_{k}(t)
:= \sum_{\ell \in \mathcal{L}} f (t {+} \tfrac{\ell}{N} ) \, e^{- 2\pi i c_k \ell / N}
\quad \text{for} \;\; t \in S \; (\subset [0,\tfrac{1}{N})) .
\end{equation}
Since $E(\Omega)$ is a Riesz basis, thus a frame, for $L^2(S)$ with lower bound $A > 0$, we use Proposition \ref{prop:RB-basic-operations} to obtain
\[
\begin{split}
\sum_{k=1}^K \sum_{\lambda \in \Omega + c_k} |\langle f, e^{2\pi i \lambda (\cdot)} \rangle_{L^2(\cup_{\ell \in \mathcal{L}} \, S{+}\frac{\ell}{N})}|^2
&= \sum_{k=1}^K \sum_{\lambda \in \Omega + c_k} |\langle F_{k}, e^{2\pi i \lambda (\cdot)} \rangle_{L^2(S)} |^2 \\
&\geq A \, \sum_{k=1}^K
\| F_{k} \|_{L^2(S)}^2 .
\end{split}
\]
Collecting \eqref{eqn:def-Fm} for $k=1,\ldots,K$, we get
\begin{equation}\label{eqn:pf-frame-rel}
\big[
F_{k}(t)
\big]_{k=1}^K
=
\underbrace{\Big[
 e^{- 2\pi i c_k \ell / N}
\Big]_{c_k \in \mathcal{K}, \, \ell \in \mathcal{L}}}_{W_{\mathcal{K},\mathcal{L}} \,\in\, \C^{K \times L}} \,
 \big[ f ( t {+} \tfrac{\ell}{N} )
  \big]_{\ell \in \mathcal{L}}
\quad \text{for} \;\; t \in S .
\end{equation}
(i) If $K < L$, then the kernel of $W_{\mathcal{K},\mathcal{L}} \in \C^{K \times L}$ has dimension at least $L {-} K > 0$.
Fix any nontrivial vector $\bv \in \kernel ( W_{\mathcal{K},\mathcal{L}} ) \subset \C^{L}$ and set $f \in L^2(\cup_{\ell \in \mathcal{L}} \, S{+}\frac{\ell}{N})$ by $[ f ( t {+} \frac{\ell}{N} ) ]_{\ell \in \mathcal{L}} = \bv$ for $t \in S$, so that $\sum_{k=1}^K \| F_{k} \|_{L^2(S)}^2 = 0$ by \eqref{eqn:pf-frame-rel} while $\| f \|_{L^2(\cup_{\ell \in \mathcal{L}} \, S{+}\frac{\ell}{N})}^2 = |S| \cdot \| \bv \|_2^2 \neq 0$.
Hence, the system $E(\cup_{k=1}^K \Omega {+} c_k)$ is not a frame for $L^2(\cup_{\ell \in \mathcal{L}} \, S{+}\frac{\ell}{N})$.
Note that since $K < L$, the matrix $W_{\mathcal{K},\mathcal{L}} \in \C^{K \times L}$ is not injective. \\
(ii) If $K \geq L$, then
\[
\begin{split}
\sum_{k=1}^K |F_{k}(t)|^2
&\overset{\eqref{eqn:pf-frame-rel}}{=} \Big\| W_{\mathcal{K},\mathcal{L}} \,  \big[ f ( t {+} \tfrac{\ell}{N} ) \big]_{\ell \in \mathcal{L}} \Big\|_2^2 \\
& \overset{(K \geq L)}{\geq}
\sigma_{\min}^2 (W_{\mathcal{K},\mathcal{L}})  \cdot \sum_{\ell \in \mathcal{L}} | f ( t {+} \tfrac{\ell}{N} ) |^2
\quad \text{for} \;\; t \in S ,
\end{split}
\]
and integrating both sides with respect to $t$ gives
\[
\sum_{k=1}^K
\| F_{k} \|_{L^2(S)}^2
\;\geq\;
\sigma_{\min}^2 (W_{\mathcal{K},\mathcal{L}}) \cdot \| f \|_{L^2(\cup_{\ell \in \mathcal{L}} \, S{+}\frac{\ell}{N})}^2 ,
\]
and therefore
\[
\sum_{k=1}^K \sum_{\lambda \in \Omega + c_k} \big| \langle f, e^{2\pi i \lambda (\cdot)} \rangle_{L^2(\cup_{\ell \in \mathcal{L}} \, S{+}\frac{\ell}{N})} \big|^2
\;\geq\;
A \, \sigma_{\min}^2 (W_{\mathcal{K},\mathcal{L}})
\cdot \| f \|_{L^2(\cup_{\ell \in \mathcal{L}} \, S{+}\frac{\ell}{N})}^2 .
\]
This shows that if $K \geq L$, then $E(\cup_{k=1}^K \Omega {+} c_k)$ is a frame for $L^2(\cup_{\ell \in \mathcal{L}} \, S{+}\frac{\ell}{N})$ if and only if the (tall) matrix $W_{\mathcal{K},\mathcal{L}} \in \C^{K \times L}$ is injective.

\medskip

\noindent
\textbf{(b) Riesz sequence}. \
For any $\boldb = \{ b_\lambda \} \in \ell_2 (\cup_{k=1}^K \Omega {+} c_k)$, we set
\[
f = \sum_{\lambda \in \cup_{k=1}^K \Omega + c_k} b_\lambda \, e^{2\pi i \lambda (\cdot)} .
\]
Then for $t \in [0,\frac{1}{N})$ and $\ell \in \{ 0, 1, \cdots, N{-}1 \}$, we have
\[
f ( t {+} \tfrac{\ell}{N} )
= \sum_{k=1}^K \, \sum_{\lambda \in \Omega + c_k}
b_\lambda \, e^{2\pi i \lambda (t + \frac{\ell}{N} )}
= \sum_{k=1}^K e^{2\pi i c_k \ell / N} \sum_{\lambda \in \Omega + c_k} b_\lambda \, e^{2\pi i \lambda t} .
\]
Collecting this equation for $\ell \in \mathcal{L}$, we obtain
\begin{equation}\label{eqn:pf-RS-rel}
\big[
f ( t {+} \tfrac{\ell}{N} )
\big]_{\ell \in \mathcal{L}}
= \underbrace{\Big[
e^{2\pi i c_k \ell / N}
\Big]_{\ell \in \mathcal{L}, \, k=1,\ldots, K}}_{(W_{\mathcal{K},\mathcal{L}})^* \,\in\, \C^{L \times K}}
 \big[ \textstyle \sum_{\lambda \in \Omega + c_k} b_\lambda \, e^{2\pi i \lambda t}
 \big]_{k=1}^K .
\end{equation}
(i) If $K > L$, then the kernel of $(W_{\mathcal{K},\mathcal{L}})^* \in \C^{L \times K}$ has dimension at least $K {-} L > 0$.
Fix any nontrivial vector $\bv = \{ v_k \}_{k=1}^K \in \kernel ( ( W_{\mathcal{K},\mathcal{L}} )^* ) \subset \C^{K}$.
For each $k = 1, \ldots, K$, the system $E(\Omega {+} c_k)$ is a Riesz basis for $L^2(S)$, so there exists a unique sequence $\{ b_\lambda \} \in \ell_2 (\Omega {+} c_k)$ such that
$\sum_{\lambda \in \Omega + c_k} b_\lambda \, e^{2\pi i \lambda t} = v_k$ for a.e.~$t \in S$.
We thus obtain a nontrivial sequence $\boldb = \{ b_\lambda \} \in \ell_2 (\cup_{k=1}^K \Omega {+} c_k)$ satisfying
\[
\big[ \textstyle \sum_{\lambda \in \Omega + c_k} b_\lambda \, e^{2\pi i \lambda t}
\big]_{k=1}^K
=
\bv
\quad \text{for a.e.} \;\; t \in S .
\]
The corresponding function $f$ satisfies $\| f \|_{L^2(\cup_{\ell \in \mathcal{L}} \, S{+}\frac{\ell}{N})}^2 = \int_S \sum_{\ell \in \mathcal{L}}  | f ( t {+} \frac{\ell}{N} )|^2 \, dt = 0$ by \eqref{eqn:pf-RS-rel} while $\| \boldb \|_2 \neq 0$.
Hence, the system $E(\cup_{k=1}^K \Omega {+} c_k)$ is not a Riesz sequence in $L^2(\cup_{\ell \in \mathcal{L}} \, S{+}\frac{\ell}{N})$.
Note that since $K > L$, the matrix $W_{\mathcal{K},\mathcal{L}} \in \C^{K \times L}$ is not surjective. \\
(ii) If $K \leq L$, then
\[
\begin{split}
\sum_{\ell \in \mathcal{L}}  \big| f ( t {+} \tfrac{\ell}{N} ) \big|^2
&\overset{\eqref{eqn:pf-RS-rel}}{=} \Big\| (W_{\mathcal{K},\mathcal{L}})^* \,  \big[ \textstyle \sum_{\lambda \in \Omega + c_k} b_\lambda \, e^{2\pi i \lambda t}
 \big]_{c_k \in \mathcal{K}} \Big\|_2^2 \\
&\overset{(K \leq L)}{\geq}
\sigma_{\min}^2 (W_{\mathcal{K},\mathcal{L}})  \cdot \sum_{k=1}^K \Big| \sum_{\lambda \in \Omega + c_k} b_\lambda \, e^{2\pi i \lambda t}  \Big|^2
\quad \text{for} \;\; t \in S ,
\end{split}
\]
where we used that $\sigma_{\min}^2 ((W_{\mathcal{K},\mathcal{L}})^* ) = \sigma_{\min}^2 (W_{\mathcal{K},\mathcal{L}})$.
Further, integrating both sides with respect to $t$ gives
\[
\begin{split}
\| f \|^2_{L^2(\cup_{\ell \in \mathcal{L}} \, S{+}\frac{\ell}{N})}
&\geq
\sigma_{\min}^2 (W_{\mathcal{K},\mathcal{L}})
\cdot \sum_{k=1}^K \Big\| \sum_{\lambda \in \Omega + c_k} b_\lambda \, e^{2\pi i \lambda (\cdot)} \Big\|_{L^2(S)}^2 \\
&\geq
\sigma_{\min}^2 (W_{\mathcal{K},\mathcal{L}})
\cdot \sum_{k=1}^K  A \, \big\| \{ b_\lambda \}_{\lambda \in \Omega + c_k} \big\|_2^2  \\
&= A \, \sigma_{\min}^2 (W_{\mathcal{K},\mathcal{L}}) \cdot \| \boldb \|_2^2 ,
\end{split}
\]
where we used that $E(\Omega)$ is a Riesz basis for $L^2(S)$ with lower Riesz bound $A > 0$.
This shows that if $K \leq L$, then $E(\cup_{k=1}^K \Omega {+} c_k)$ is a Riesz sequence in $L^2(\cup_{\ell \in \mathcal{L}} \, S{+}\frac{\ell}{N})$ if and only if the (short and fat) matrix $W_{\mathcal{K},\mathcal{L}} \in \C^{K \times L}$ is surjective.

\medskip

\noindent
\textbf{(c) Riesz basis}. \
Recall that a sequence in a separable Hilbert space is a Riesz basis if and only if it is both a frame and a Riesz sequence.
Hence, this part follows immediately by combining the \emph{frame} and \emph{Riesz sequence} parts.
\hfill $\Box$ 


\section*{Acknowledgments}

This work was supported by the National Research Foundation of Korea (NRF) grant funded by the Korean government (MSIT) (RS-2023-00275360).
The proof of Theorem \ref{thm:two-sets-nested-translation-by-a} is essentially due to David Walnut.
The author would like to thank Andrei Caragea, G\"{o}tz Pfander and David Walnut for fruitful and stimulating discussions.
Theorem~\ref{thm:two-sets-nested-translation-by-a-dimension-d} is inspired by a personal communication with Laura De Carli. The author would like to thank her for helpful comments. 

%






\end{document}